\documentclass[12pt]{article}
\usepackage[reqno]{amsmath}
\usepackage{amssymb, enumerate}
\usepackage{graphicx}
\usepackage{tikz}
\usetikzlibrary{decorations.markings}
\usepackage{array}
\usepackage{float}
\usepackage{wrapfig}
\usepackage{fullpage,url}
\usepackage{nicefrac,relsize}
\usepackage{listings,doi}

\date{\today}
\expandafter\ifx\csname pdfpagebox\endcsname\relax\else
\fi

\usepackage{amsthm}
\newtheorem{theorem}{Theorem}[section]
\newtheorem{lemma}[theorem]{Lemma}
\newtheorem{prop}[theorem]{Proposition}
\newtheorem{corollary}[theorem]{Corollary}

\newcommand\sqr[2]{{\vbox{\hrule height.#2pt
    \hbox{\vrule width.#2pt height#1pt \kern#1pt
        \vrule width.#2pt}\hrule height.#2pt}}}
\renewcommand\qed{%
	\ifmmode\eqno\sqr53
	\else\nolinebreak\ \hfill\sqr53\medbreak\fi}

\newcommand\cH{{\mathcal H}}
\newcommand\cI{{\mathcal I}}

\newcommand\cL{{\mathcal L}}

\newcommand\cP{{\mathcal P}}

\newcommand\G{{\Gamma}}

\newcommand\fld{{\mathbb F}}

\newcommand\bmat[1]{\begin{bmatrix} #1 \end{bmatrix}}

\newcommand\mat[3]{\mathrm{Mat}_{#1\times #2}(#3)}
\DeclareMathOperator\Aut{Aut}
\DeclareMathOperator\Sym{Sym}

\DeclareMathOperator{\PG}{PG}

\DeclareMathOperator{\GQ}{GQ}
\DeclareMathOperator{\pg}{pg}
\DeclareMathOperator{\PGaL}{P\Gamma L}

\title{Pseudo-geometric strongly regular graphs\\ with a regular point}
\author{Edwin R.~van Dam\footnote{Dept. Econometrics and OR, Tilburg University, the Netherlands. email: \protect\url{edwin.vandam@uvt.nl}}, Krystal Guo\footnote{Korteweg-de Vries Institute, University of Amsterdam, Amsterdam, the Netherlands. email: \protect\url{k.guo@uva.nl}}}

\date{April 9, 2026}

\begin{document}
\maketitle

\begin{abstract}
We study pseudo-geometric strongly regular graphs whose second subconstituent with respect to a vertex is a cover of a strongly regular graph or a complete graph. We call such a vertex a \textsl{regular point}, thereby generalizing regular points in generalized quadrangles. We characterize all graphs containing a  regular point, and use our characterization to find many new strongly regular graphs. Thereby, we give a partial answer to a question posed by Gardiner, Godsil, Hensel, and Royle. We give an explicit construction for $q$ new, pairwise non-isomorphic graphs with the same parameters as the collinearity graph of generalized quadrangles of order $(q,q)$ and a new non-geometric graph with the same parameters as the collinearity graph of the Hermitian generalized quadrangle of order $(q^2,q)$, for prime powers $q$. 
 Using our characterization, we computed $133\,718$  new strongly regular graphs with parameters $(85,20,3,5)$ and $27\,298$ strongly regular graphs with parameters $(156,30,4,6)$.
 
 
\end{abstract}

\section{Introduction}\label{sec:intro}

We study pseudo-geometric strongly regular graphs (i.e., with the same parameters as those coming from generalized quadrangles) whose second subconstituent with respect to a vertex is a cover of a strongly regular graph or a complete graph. We call such a vertex a \textsl{regular point}, thereby generalizing regular points in generalized quadrangles. We characterize all graphs containing a  regular point, and use our characterization to find many new strongly regular graphs. Thereby, we give a partial answer to a question posed by Gardiner, Godsil, Hensel, and Royle \cite{GARDINER1992161}. We give an explicit construction for $q$ new, pairwise non-isomorphic graphs with the same parameters as the collinearity graph of generalized quadrangles of order $(q,q)$ and a new non-geometric graph with the same parameters as the collinearity graph of the Hermitian generalized quadrangle of order $(q^2,q)$, for prime powers $q$.

Strongly regular graphs were introduced by Bose \cite{Bose1963} and form an important area of study in algebraic graph theory and coding theory; see the recent monograph by Brouwer and Van Maldeghem \cite{BvMSRG}. The automorphism group $G$, of a graph $\Gamma$ acts on the pairs of vertices of $\Gamma$; this action has at least three orbits, corresponding to pairs of vertices that are equal, adjacent   and non-adjacent, respectively. If there are exactly three orbits, the graph is said to be \textsl{rank 3}. Strongly regular is a relaxation of this symmetry property; a graph is \textsl{strongly regular} with parameters $(\nu,k,\lambda,\mu)$ if it is a $k$-regular graph on $\nu$ vertices, any two adjacent vertices have precisely $\lambda$ common neighbours and any two non-adjacent vertices have precisely $\mu$ common neighbours. Every rank $3$ graph is strongly regular. Though the converse is not true, many strongly regular graphs arising from algebraic constructions are rank $3$. 

Much, relatively recent, work has been done on the construction of new families of strongly regular graphs. For example, 
Crnkovi\'{c}, \v{S}vob, and Tonchev \cite{CrnkovicDean2021Srgw} found twelve new strongly regular graphs with parameters $(81, 30, 9, 12)$ using automorphisms of the previously known graphs. In addition, this search led to the discovery of a new partial geometry $\pg(5,5,2)$, that was independently constructed by Kr\v{c}adinac \cite{KRCADINAC2021105493}. Feng and Xiang \cite{FENG2012982} and Feng, Momihara, and Xiang \cite{FenMomXia2014} gave a construction of strongly regular graphs using cyclotomic classes and skew Hadamard different sets, thereby extending the fundamental work of Van Lint and Schrijver \cite{VanLintSchrijver}.  
Several other methods of finding new strongly regular graphs with switching operations have been recently explored by Ihringer and Munemasa \cite{IHRINGER2019464} and Ihringer, Pavese, and Smaldore \cite{ IHRINGER2021112560}. 
Less recently, but closer to our results, Brouwer, Ivanov, and Klin \cite{Bro1989} constructed an infinite family of strongly regular graphs whose second subconstituent is again strongly regular.

We observed that the second subconstituent, $\Gamma_2$, of the collinearity graph of the symplectic quadrangle $W(3)$ with respect to any vertex is a distance-regular antipodal $3$-cover of $K_{9}$. There are nine other strongly regular graphs $\G$ cospectral with the collinearity graph of $W(3)$, which have a vertex $v$ such that the second subconstituent of $\G$ with respect to $v$ is isomorphic to $\Gamma_2$. This led us to believe that we can construct strongly regular graphs cospectral with the collinearity graph of $W(q)$ for prime powers $q$, by extending the second subconstituent of $W(q)$ appropriately. We succeed in this endeavour in more generality and obtain a characterization of pseudo-geometric strongly regular graphs having a so-called regular point in Theorem \ref{thm:characterizewithscheme}. 

The second subconstituents of strongly regular graphs were first extensively studied in \cite{CAMERON1978257}. 
Gardiner proved that the second neighbourhood of any vertex in a Moore graph of diameter two is an antipodal distance-regular graph of diameter three \cite{Gar1974}. 
Further, Gardiner et al.~\cite{GARDINER1992161} show that if every
vertex in a strongly regular graph $G$ has an antipodal distance-regular graphs of diameter three as its second subconstituent, then $G$ is the noncollinearity graph of a
special type of semipartial geometry. 

Cameron, Goethals and Seidel showed in \cite{CAMERON1978257} that any cospectral mate of a generalized quadrangle of order $(q,q^2)$ is geometric; that is, every strongly regular graph with the same parameters as the collinearity graph of a generalized quadrangle of order $(q,q^2)$ is the collinearity graph of such a generalized quadrangle. 
Our results show that the collinearity graphs of generalized quadrangles of orders $(q,q)$ and $(q^2,q)$ are different, that is, there are strongly regular graphs with the same parameters that are not geometric.

We have organized our results as follows. In Section \ref{sec:prelim}, we give necessary definitions and recall theorems and concepts we need throughout the paper. In Section \ref{sec:4class}, we give a characterization, in Theorem \ref{thm:characterizewithscheme}, of all graphs with the same parameter set as a generalized quadrangle of order $(s,t)$, where $s\geq t$, with a regular point. In Section \ref{sec:newgraphs}, we describe how we can use this characterization to construct new strongly regular graphs from known ones with a regular point. In Section \ref{sec:HGQ}, we apply this to the Hermitian generalized quadrangle of order $(q^2,q)$.
In Section \ref{sec:2nbhdextensions}, we specify to the generalized quadrangles of order $(q,q)$, in particular the symplectic quadrangle $W(q)$,
and answer a question posed by Gardiner et al.~\cite{GARDINER1992161} about graphs with the same parameters as $W(q)$, where one (but not necessarily all) vertex is a regular point, by giving the full characterization. We use this to give explicit constructions of such graphs in Section \ref{sec:construction}; in particular, in Theorem \ref{thm:Wqmates} we show that the symplectic generalized quadrangle $W(q)$ has at least $q$ (pairwise non-isomorphic) cospectral mates, for $q \geq 3$. Using this construction, we carried out computations in SageMath \cite{sage}, which we describe in Section \ref{sec:computations}; in particular, we computed $133\,718$  new isomorphism classes of strongly regular graphs with parameters $(85,20,3,5)$ and $27\,298$ new isomorphism classes of strongly regular graphs with parameters $(156,30,4,6)$. A database of the new graphs can be found in \cite{newgrs}. We end with some final observations and future directions in Section \ref{sec:end}. Among others, we there give a more geometric characterization of regular points in Theorem \ref{thm:simplescharacterization}, and describe how we can use our results to construct also new pseudo-geometric strongly regular graphs with the same parameters as the collinearity graph of a generalized quadrangle of order $(q-1,q+1)$, and related antipodal distance-regular $q$-covers of the complete graph on $q^2$ vertices in Section \ref{sec:q-1q+1}. 

\section{Preliminaries}\label{sec:prelim}

A \textsl{generalized quadrangle of order $(s,t)$}, denoted $\GQ(s,t)$, is a point-line incidence structure such that
\begin{enumerate}[(i)]
	\item each point is incident with $t+1$ lines and any two points are incident with at most one line;
	\item each line is incident  with $s+1$ points and any two lines are incident with at most one point; and
	\item for every point $x$ and line $L$ not incident with $x$, there exists a unique pair $(y,M)$ such that $M$ is incident with both $x$ and $y$ and $y$ is incident with $L$. 
\end{enumerate}
For background on generalized quadrangles, we refer to \cite{PT}. The \textsl{collinearity graph} (or \textsl{point graph}) of a generalized quadrangle is the graph whose vertices are the points of the generalized quadrangle and where two point are adjacent if there exists a line incident with both points. The collinearity graph of a generalized quadrangle is strongly regular with parameters
\[
((s+1)(st+1), s(t+1), s-1, t+1).
\]
When the context is clear, we speak of the generalized quadrangle and its collinearity graph exchangeably. 
For a graph $\G$ and a vertex $x$ of $\G$, we will denote by $\G_i(x)$ the \textsl{$i$th subconstituent} or \textsl{$i$th neighbourhood} of $x$, which is the graph induced by the set of vertices at distance $i$ from $x$. 

By way of example, we will give the construction of the \textsl{symplectic generalized quadrangle}, denoted $W(q)$, whose points are the points of $\PG(3,q)$ and whose lines are a subset of the lines of $\PG(3,q)$. Let $q$ be a prime power. Following \cite[Section 5.5]{GR}, let $S \in \mat{4}{4}{\fld_q}$ be the following matrix:
\[ S = \bmat{0 & 1 & 0 & 0 \\
 -1 & 0 &0 &0 \\
 0&0&0&1 \\
 0 & 0 &-1&0}.
 \]
 A line of $\PG(3,q)$ is said to be \textit{totally isotropic} if $u^{\top} S v = 0$ for all $u,v$ on the line. Let $W(q)$ be the point-line incidence structure whose points are the points of $\PG(3,q)$ and lines are the isotropic lines of $\PG(3,q)$. We have that $W(q)$ is a generalized quadrangle of order $(q,q)$. 
Figure \ref{fig:gq22} shows the collinearity graph of $W(2)$. 

\begin{figure}[htbp]
    \centering
    \includegraphics[scale=0.5]{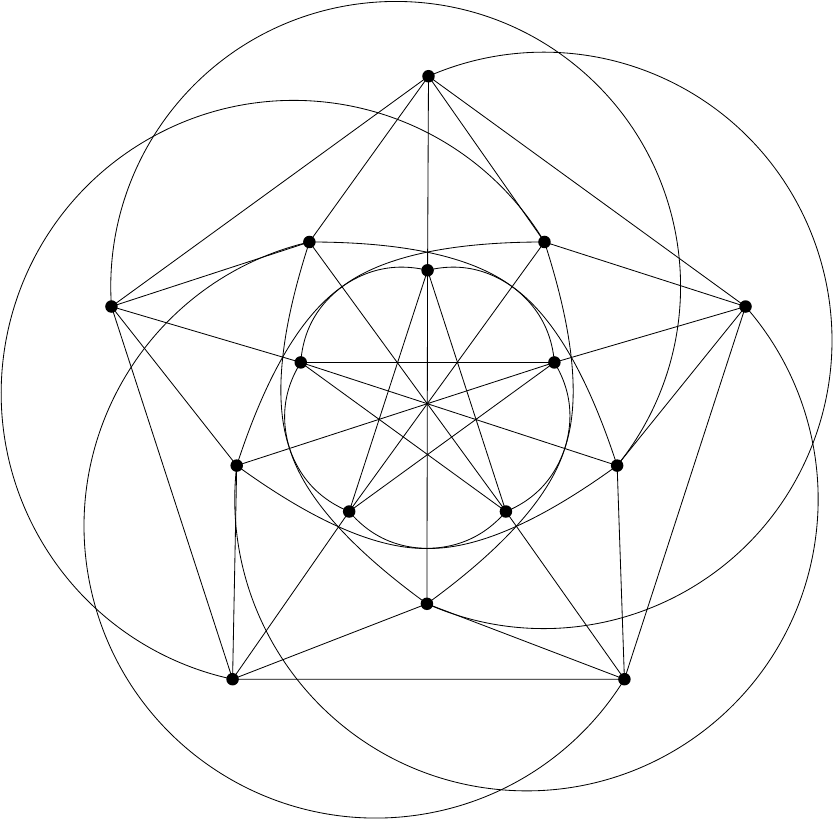} 
    \caption{The collinearity graph of $W(2)$, the unique $\GQ(2,2)$, up to isomorphism. It is a strongly regular graph with parameters $(15,6,1,3)$.     \label{fig:gq22}}
\end{figure}

Another generalized quadrangle which we look at in this paper is the \textsl{Hermitian generalized quadrangle} which we will denote $H(3,q^2)$. Let $H$ be a nonsingular Hermitian variety of the projective space $\PG(3,q^2)$, where $q$ is a prime power. Then the points and lines of $H$ give a $\GQ(q^2,q)$, see \cite[Chapter 3]{PT}.

For a set of points $X$ in a generalized quadrangle $\GQ(s,t)$, we denote by $X^{\perp}$ the set of points that are collinear to all points in $X$ (note that a point is collinear to itself). The \textsl{span} of  points $x,y$, is the set $\{x,y\}^{\perp \perp}$. 

If $x,y$ are distinct, collinear points, or if $x,y$ are not collinear and $| \{x,y\}^{\perp \perp}| = t+1$, then the pair $(x,y)$ is said to be \textsl{regular}. A point $x$ is regular if $(x,y)$ is regular for all points $y\neq x$. 

In the collinearity graph, a regular point $v$ is a vertex such that, for all $x \in \G_2(v)$, the common neighbours of $x$ and $v$ have $t+1$ common neighbours. Figure \ref{fig:gq22-regpt} shows the collinearity graph of $W(2)$, with a regular pair, $(x,y)$. 
\begin{figure}[htbp]
    \centering
    \includegraphics[scale=0.6]{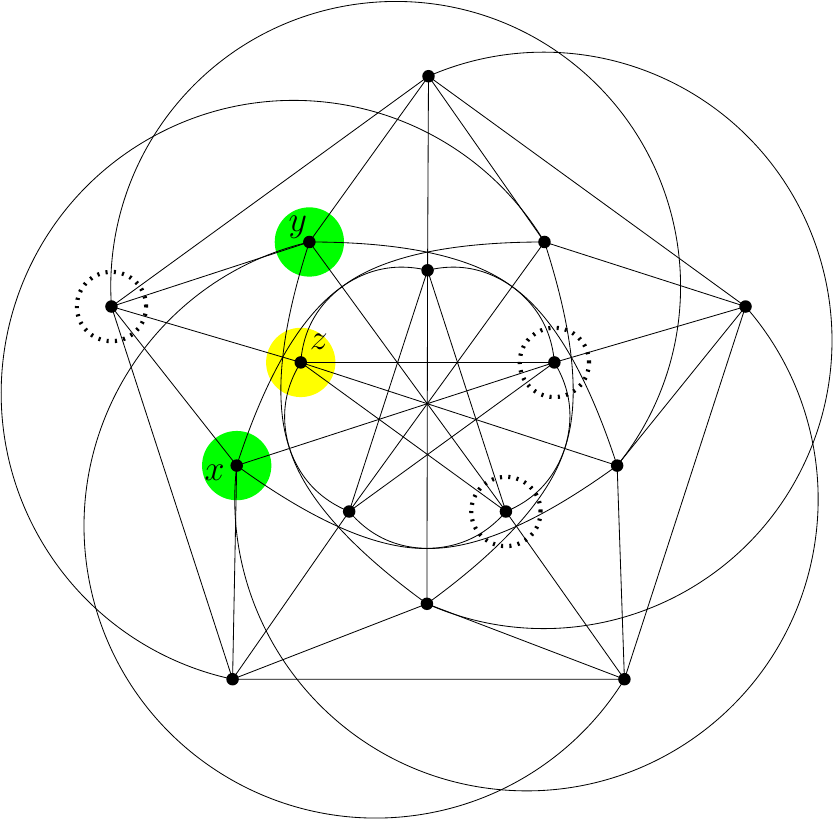}
    \caption{The collinearity graph of $W(2)$, with two vertices $x,y$ highlighted in green. The common neighbours of $x$ and $y$ are shown with dotted circles. The span of $x$ and $y$,  $| \{x,y\}^{\perp \perp}| = t+1$, consists of $x,y$ together with $z$, the vertex highlighted in yellow.    \label{fig:gq22-regpt}}
\end{figure}

For example, Figure \ref{fig:gq22-nbhds} shows the subconstituents of the collinearity graph of $W(2)$, with respect to vertex $u$; the second subconsituent of $W(2)$ is a copy of the cube graph. 

\begin{figure}
    \centering
    \includegraphics[scale=0.38]{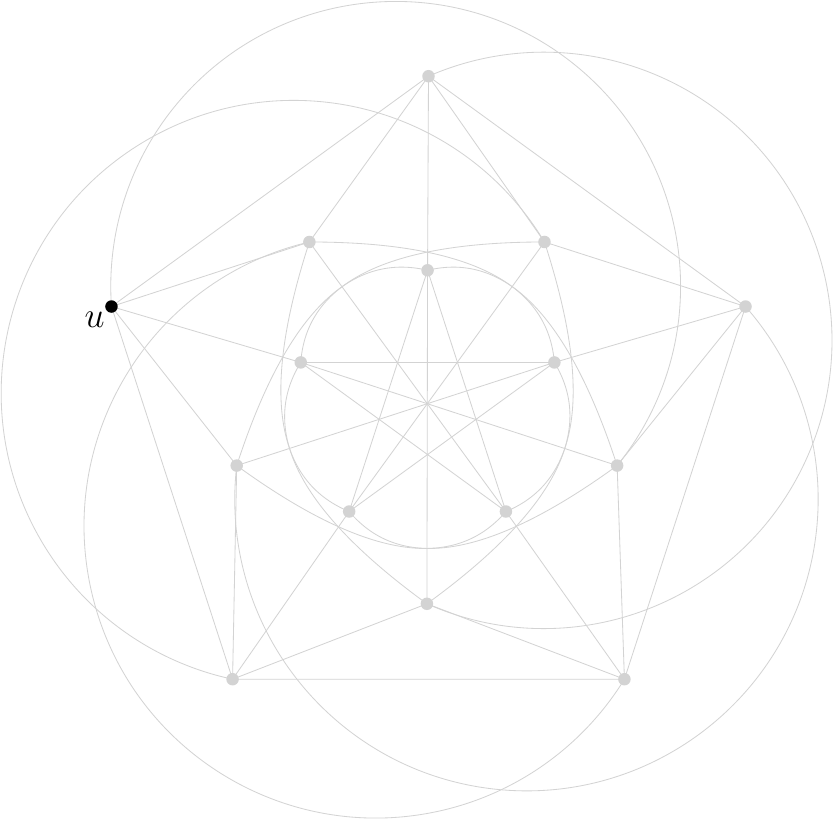}
      \includegraphics[scale=0.38]{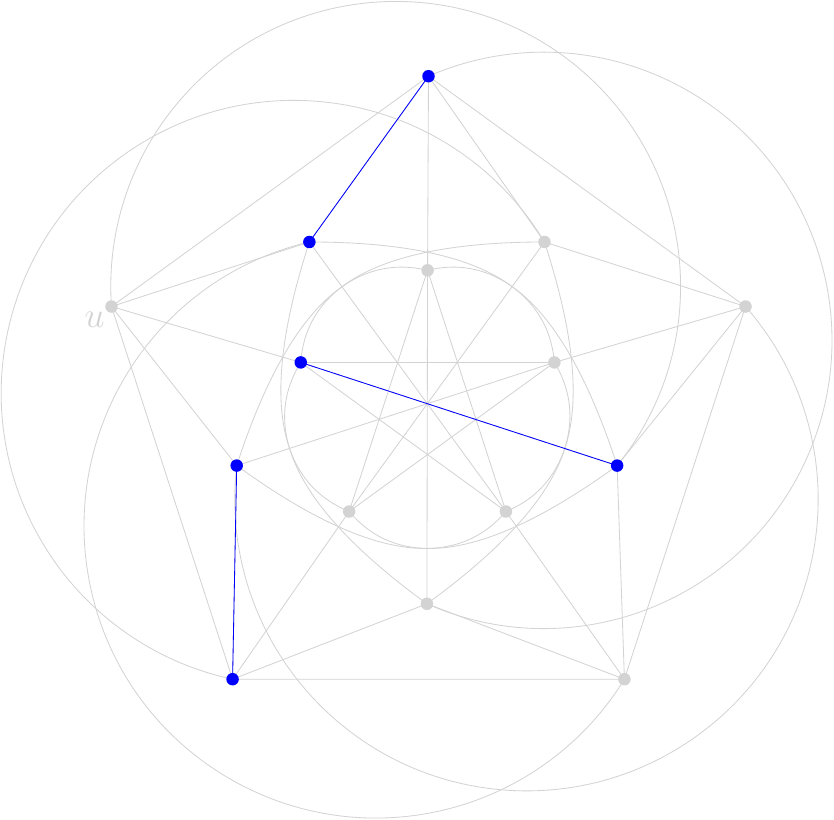}
        \includegraphics[scale=0.38]{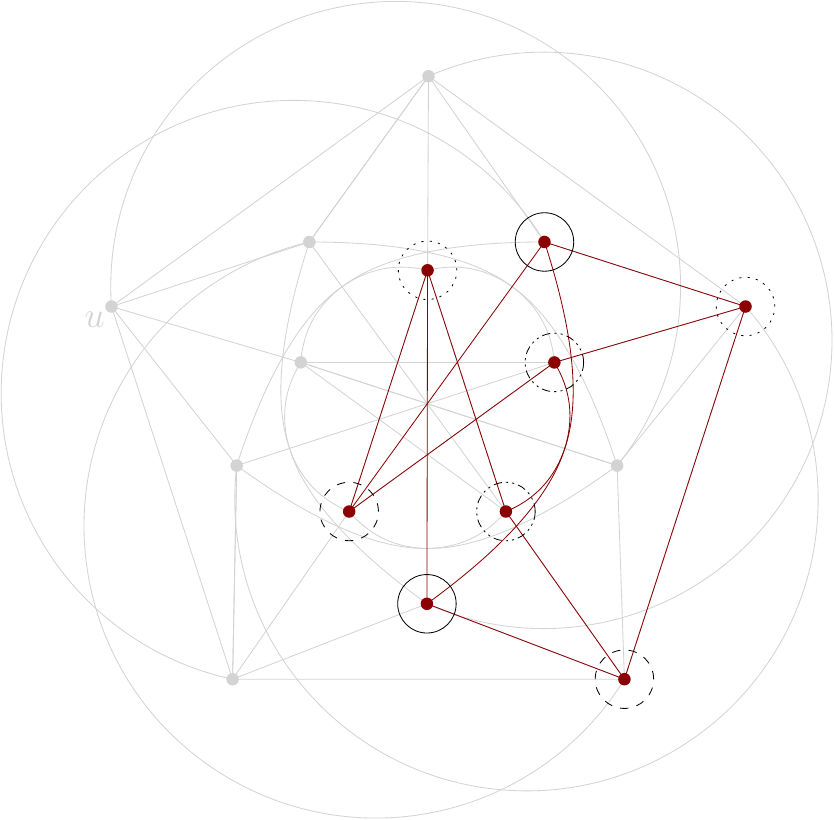}
    \caption{The subconstituents of the collinearity graph of $W(2)$, with respect to vertex $u$. The second subconsituent is isomorphic to the cube graph, which is an antipodal, distance-regular $2$-cover of $K_4$; its fibres are shown by the circles surrounding the vertices.   \label{fig:gq22-nbhds}}
\end{figure}

As a final preliminary result, we mention Godsil-McKay switching, as presented in the following lemma, which is a common way of constructing cospectral graphs.

\begin{lemma} \cite[Godsil-McKay switching]{GMswitching}.\label{lem:GM}
Let $\G$ be a graph and let $\Pi=\{D,C_1,\ldots,C_m\}$
be a partition of the vertex set of $\G$ such that $\{C_1,\ldots,C_m\}$ is an equitable partition in $\G$.
Suppose that for all $i$, each vertex in $D$ is adjacent to none, half, or all of the vertices in $C_i$.
Let $\G'$ be obtained from $\G$ as follows.
For all $i$ and each vertex in $D$ that is adjacent to half of the vertices in $C_i$,
delete the corresponding $\frac{1}{2}|C_i|$ edges and
join the vertex instead to the $\frac{1}{2}|C_i|$ other vertices in $C_i$.
Then $\G$ and $\G'$ have the same spectrum.
\end{lemma}

For definitions and background on association schemes and distance-regular graphs, we refer to the monographs by Brouwer, Cohen, and Neumaier \cite{BCN}, Brouwer and Van Maldeghem \cite{BvMSRG}, or Godsil and Royle \cite{GR}.

\section{Regular points in pseudo-geometric graphs}
\label{sec:4class}

In this section, we will generalize regular points in generalized quadrangles to the corresponding pseudo-geometric graphs, by considering $4$-class imprimitive association schemes on the second subconstituent.
In the main theorem, Theorem \ref{thm:characterizewithscheme}, we characterize all graphs with the same parameter set as a generalized quadrangle of order $(s,t)$, where $s\geq t$, with a regular point. We will then use this characterization to construct new pseudo-geometric graphs, in particular for the Hermitian generalized quadrangle.

\subsection{A four-class association scheme}

Consider a regular point $v$ in a generalized quadrangle of order $(s,t)$ with $s \geq t \geq 2$. If we remove $v$ from the so-called hyperbolic lines $\{v,x\}^{\perp \perp}$, then these sets (of size $t$) partition the points in the second subconstituent $\Gamma_2(v)$ of the collinearity graph $\G$, and two points are in the same hyperbolic line if and only if they are at distance $3$ in (the induced graph on) $\G_2(v)$. In this graph, there are two kinds of points at distance $2$ from a given point $x$: those that are collinear to some other point on the hyperbolic line through $x$, and those that are not. This describes a $4$-class imprimitive association scheme on the second subconstituent if $s>t$, as was shown by Ghinelli and L\"{o}we \cite{GHINELLI199587}, who also characterized generalized quadrangles by such a scheme. The first eigenmatrix $P$ of the scheme is as in \eqref{eq:eigenmatrix} below. In the extremal case that $s=t$, one of the distance-two relations is void (there are none of ``those that are not''), and one obtains a metric $3$-class association scheme corresponding to a distance-regular antipodal cover of a complete graph. 

We note that the hyperbolic lines (minus $v$) form the so-called fibres of the scheme.
Between any two such fibres there is either a matching or nothing; the induced quotient scheme on the fibres is a $2$-class association scheme corresponding to a strongly regular Latin square graph with parameters $(s^2,(t+1)(s-1),t^2-t+s-2,t(t+1))$ (except for $s=t$, when it is a $1$-class scheme and the graph is complete).

This strongly regular graph in fact corresponds to an incidence structure between the first and second subconstituents that will play a crucial role in the construction of our new strongly regular graphs.
As points of this incidence stucture we take the $s^2$ fibres, and as blocks we take the $(t+1)s$ vertices in $\Gamma_1(v)$. These blocks are naturally partitioned into $t+1$ groups of size $s$ (the lines through the regular point $v$). A point is incident to a block whenever the block is adjacent to the vertices in the fibre.

It follows that each point is incident to $t+1$ blocks, one from each group, and each block is adjacent to $s$ points.
Two blocks from different groups share at most one point, so we have a partial linear space. Because two blocks from the same group are not incident to a common point, each group of blocks induces a partition of the point set, so we have in fact a $(t+1)$-net of order $s$. Other terminology for this is a partial geometry $\pg(s-1,t,t$) or an affine resolvable $1$-$(s^2,s,t+1)$ design. In the extremal case that $s=t$, this is in fact an affine plane. It is also equivalent to a set of $t-1$ mutually orthogonal Latin squares, and this corresponds exactly to the above strongly regular Latin square graph. 

\subsection{A characterization of regular points} 

We now generalize the above to the corresponding pseudo-geometric graphs, that is, to strongly regular graphs with parameters $((s+1)(st+1),s(t+1),s-1,t+1)$. In some of the literature, e.g. \cite{GUO2020103128}, these are also called pseudo-generalized quadrangles. We call a vertex in such a graph $\G$ a regular point if the induced graph on the second subconstituent generates an imprimitive $4$-class association scheme  $\cH$ with (first) eigenmatrix 
\begin{equation}\label{eq:eigenmatrix}
P=\begin{bmatrix} 1 & (s-1)(t+1) & (s-1)(t^2-1) & t-1 & (s-1)t(s-t)\\
1 &s-1 & 1-s &-1 & 0\\
1 &s-t-1 & (t-1)(s-t-1) &t-1 & -t(s-t)\\
1 &-t-1 & t+1 &-1 & 0\\
1 &-t-1 & 1-t^2 &t-1 & t^2
\end{bmatrix}.
\end{equation}

We recall that we assume that $s \geq t$, where we allow the degenerate case that $s=t$; in this case we actually have a $3$-class scheme, because the final relation ($B_4$ below) is empty. All of below arguments remain valid in this extremal case. 

The above definition may look quite restrictive, and it does not resemble the definition of regular points in generalized quadrangles. It does however lead us straight to our main results below. Later on, in Theorem \ref{thm:simplescharacterization}, we will show that regular points are indeed very similar to the ones in generalized quadrangles.

We will denote the adjacency matrices (and for convenience also the corresponding relations) in the scheme by $B=B_1$, $B_2$, $B_3$, and $B_4$, respectively. All intersection numbers of the scheme follow from the eigenmatrix $P$, and it follows that the distance-$2$ graph of $B$ splits into the relations $B_2$ and $B_4$, as the number of common neighbours varies between $p^2_{11}=t$ and $p^4_{11}=t+1$, respectively. In particular, we obtain the equation
\begin{equation}\label{eq:B^2}
B^2=(s-1)(t+1)I+(s-2)B+tB_2+(t+1)B_4.
\end{equation}

We note that the distance-$3$ graph $B_3$ is a disjoint union of $t$-cliques; these are the fibres of the scheme. Thus, after appropriate reordering of the vertices, we can write $B_3=I \otimes J_t-I$ (by $J$ we denote an all-ones matrix). In fact, the quotient scheme on the fibres is a $2$-class scheme with matrices $\hat{B}$ and $\hat{B_4}$, and we can write $B+B_2=\hat{B} \otimes J_t$ and $B_4=\hat{B_4} \otimes J_t$.
More specifically, $\hat{B}$ is a strongly regular graph on $s^2$ vertices with spectrum $\{(s-1)(t+1)^{[1]},s-t-1^{[(s-1)(t+1)]},-t-1^{[(s-1)(s-t)]}\}$ (again, note that there is degeneracy if $s=t$, in which case $\hat{B}$ is a complete graph).

In the following, we will characterize the above situation, and then exploit it to construct pseudo-geometric graphs with a regular point from generalized quadrangles with a regular point.

To begin with, we denote the adjacency matrix of the strongly regular graph $\G$ by $A$, and observe that it satisfies the equation 
\begin{equation}\label{eq:Asquared}
A^2=(s-t-2)A+(t+1)(s-1)I+(t+1)J.
\end{equation}
When we partition $A$ according to the subconstituents as
\begin{equation}\label{eq:Apartition}
A=\begin{bmatrix}
0 &j^{\top} & 0^{\top}\\
j & C & N\\
0 & N^{\top} &B
\end{bmatrix},
\end{equation}
then the quadratic equation for $A$ gives some relevant equations for the subconstituents $C$ and $B$, and the incidence matrix $N$ between them.
Besides the fact that all submatrices have constant row and column sums, we find that 
\begin{align}
C^2+ NN^{\top}&=(s-t-2)C+(t+1)(s-1)I+tJ,\label{eq:C^2N}\\
CN+NB&=(s-t-2)N+(t+1)J,\label{eq:CNB}\\
B^2+ N^{\top}N&=(s-t-2)B+(t+1)(s-1)I+(t+1)J\label{eq:B^2N}.
\end{align}

\begin{lemma}\label{lem:spectrum} $C$ is a disjoint union of $t+1$ $s$-cliques if and only if $B$ has spectrum $\{(t+1)(s-1)^{[1]},s-1^{[m_2]},s-t-1^{[(t+1)(s-1)]},[-t-1]^{[m_4]}\}$, with $m_2=s^2(t^2-1)/(s+t)$ and $m_4=t(s-1)(s^2-t)/(s+t)$.
\end{lemma}

\begin{proof}
First, note that the disjoint union of $t+1$ $s$-cliques is characterized by having spectrum $\{s-1^{[t+1]},-1^{[(t+1)(s-1)]}\}$.
Next, we note that the multiplicities of a regular connected graph with (at most) four distinct eigenvalues follow from the number of vertices and the eigenvalues \cite{VANDAM1995139}.
The above matrix equations for the subconstituents (in particular \eqref{eq:CNB}) relate the so-called restricted eigenspaces of $C$ and $B$, as already observed by Cameron, Goethals, and Seidel \cite{CAMERON1978257}. In general, if $x$ is an eigenvector of $B$ that is orthogonal to the all-ones vector $j$ with an eigenvalue $\theta$ that differs from the restricted eigenvalues $\theta_1$ and $\theta_2$ of $A$, then $Nx$ is an eigenvector of $C$ that is also orthogonal to $j$, with eigenvalue $\theta-\theta_1-\theta_2$. Similarly we have a dual statement.
Applying this with $\theta_1=s-1$ and $\theta_2=-t-1$ finishes the proof; we omit the calculations of the multiplicities. Finally, we note that the relation between the spectra of the first and second subconstituent has also been worked out explicitly using Jacobi's identity on complementary minors by De Caen \cite[Cor.~3]{DECAEN1998559}.
\end{proof}

This lemma implies that if $v$ is a regular point, then $C$ is indeed a disjoint union of $t+1$ $s$-cliques, and hence
\begin{equation}\label{eq:C^2}
C^2=(s-1)I+(s-2)C.
\end{equation}
Let us now consider a regular point $v$.
We will next derive properties of an incidence relation between the set of fibres $\cP$ (points) of $\cH$ and the set of vertices $\cL$ (blocks) in $C$. By combining \eqref{eq:B^2} with \eqref{eq:B^2N}, and \eqref{eq:C^2} with \eqref{eq:C^2N}, and using that $\sum_{i=1}^4 B_i=J-I$, we now obtain that  
\begin{align}
NN^{\top}&=tsI+t(J-I-C),\label{eq:NNT}\\
N^{\top}N&=(t+1)(B_3+I)+B+B_2\label{eq:NTN}.
\end{align}

\begin{lemma} Vertices in the same fibre of $\cH$ are adjacent to the same $t+1$ vertices in $C$. 
\end{lemma}
\begin{proof} This follows from \eqref{eq:NTN}.
\end{proof}

Thus, a vertex in $C$ is adjacent to either none or all of the vertices in a given fibre. We can thus write $N=M \otimes J_{1t}$, where $M^{\top}$ can be considered as the incidence matrix of the incidence relation $\cI$ between the set $\cP$ of fibres of $\cH$ and the set $\cL$ of vertices in $C$. Equations \eqref{eq:NNT}-\eqref{eq:NTN} now reduce to 
\begin{align}
MM^{\top}&=sI+(J-I-C),\label{eq:MMT}\\
M^{\top}M&=(t+1)I+\hat{B}\label{eq:MTM}.
\end{align}
These equations immediately imply the following.
\begin{lemma} The incidence relation $\cI$ is a $(t+1)$-net of order $s$. Each $s$-clique in $C$ is a parallel class of blocks and the collinearity graph of $\cI$ is $\hat{B}$.
\end{lemma}

In the extremal case that $s=t$, note that $\cI$ is in fact an affine plane.

The above properties completely characterize pseudo-geometric strongly regular graphs with a regular point, as we shall see.
Given an association scheme $\cH$ with eigenmatrix as in  \eqref{eq:eigenmatrix}, a $(t+1)$-net $\cI$ of order $s$, and a bijection $\phi$ from the fibres of $\cH$ to the point set $\cP$ of $\cI$, we can define a graph $\G=\G(\cH,\cI,\phi)$. The vertex set of $\G$ is
\[
\{v\} \cup \cL \cup V(\cH),
\]
where $\cL$ is the set of blocks of $\cI$ and $v$ is a vertex of origin. The adjacencies in $\G$  are as follows:
\begin{enumerate}[(i)]
    \item $v$ is adjacent to all blocks in $\cL$; 
    \item two blocks in $\cL$ are adjacent if they are parallel in $\cI$; 
    \item a block $\ell \in \cL$ is adjacent to a vertex in $V(\cH)$ if the latter is in a fibre $W$ such that $\phi(W)$ is incident to $\ell$ in $\cI$; and 
    \item two vertices of $V(\cH)$ are adjacent if they are adjacent in the first relation $B$ of $\cH$.
\end{enumerate}

\begin{lemma} If $\phi$ is an isomorphism from the quotient graph $\hat{B}$ to the collinearity graph of $\cI$, then $\G(\cH,\cI,\phi)$ is a strongly regular graph with parameters $((s+1)(st+1),s(t+1),s-1,t+1)$. 
\end{lemma}

\begin{proof}
Let $\G=\G(\cH,\cI,\phi)$. The bijection $\phi$ allows us to define the incidence matrix $M$ between blocks and fibres, i.e., we let $M_{\ell,W}=1$ precisely when $\ell$ is incident to $\phi(W)$. Then \eqref{eq:MMT} holds because of the properties of $\cI$, with $C$ as in \eqref{eq:C^2}, and \eqref{eq:MTM} holds because $\phi$ is an isomorphism. 

We then let $N=M\otimes J_{1t}$ as before, where we can arrange the columns of $N$ within fibres such that it matches the adjacency matrices of $\cH$. Then \eqref{eq:NNT} and \eqref{eq:NTN} follow. By combining these with \eqref{eq:B^2} and \eqref{eq:C^2}, we obtain \eqref{eq:C^2N} and \eqref{eq:B^2N}. 

In the meantime, we have constructed the entire adjacency matrix $A$ of $\G$ as in \eqref{eq:Apartition}. For completeness, we note that $\G$ is regular, and that $C,N,$ and $B$ have the appropriate row and column sums for $\G$ to be strongly regular with the stated parameters. In fact, all that remains to be shown is \eqref{eq:CNB}.

If we now multiply \eqref{eq:NNT} by $N$ from the right, and \eqref{eq:NTN} by $N$ from the left, compare the obtained right hand sides, and use that $N(B+B_2)=tNB$ and $N(B_3+I)=tN$, then \eqref{eq:CNB} follows.

Thus, it follows that \eqref{eq:Asquared} holds, which finishes the proof.
\end{proof}

As a result of the above lemmata, we now obtain the following characterization of pseudo-geometric strongly regular graphs  having a regular point. 

\begin{theorem}\label{thm:characterizewithscheme}
Let $s \geq t$. Then $\Gamma$ is a strongly regular graph with parameters $((ts+1)(s+1),(t+1)s,s-1,t+1)$ with a regular point if and only if $\G = \G(\cH,\cI,\phi)$, where
\begin{enumerate}[(i)]
  \item $\cH$ is an association scheme with eigenmatrix \eqref{eq:eigenmatrix}; 
\item $\cI$ is a $(t+1)$-net of order $s$;
\item $\phi$ is an isomorphism from the quotient graph $\hat{B}$ to the collinearity graph of $\cI$.
\end{enumerate}
\end{theorem}

Note that in the extremal case that $s=t$, both the collinearity graph and the quotient graph are complete graphs, hence in this case $\phi$ can be any bijection. In that case, $\cI$ is an affine plane of order $s$, and $\cH$ is a $3$-class association scheme generated by an antipodal $s$-cover of a complete graph on $s^2$ vertices. We will discuss this in more detail in Section \ref{sec:2nbhdextensions}.

\subsection{New graphs from old ones}\label{sec:newgraphs}

Using Theorem \ref{thm:characterizewithscheme}, we can now give constructions for new strongly regular graphs cospectral with a given strongly regular graph $\G$ with a regular point $v$ as follows: we may find $\cH,\cI,\phi$ as in the theorem and obtain $\widetilde{\G} = \G(\cH,\cI,\phi')$ where $\phi'$ is obtained from $\phi$ by composition with any automorphism of the collinearity graph of $\cI$. The new graph $\widetilde{\G}$ will be strongly regular with the same parameters as $\Gamma$, but not necessarily isomorphic to $\Gamma$. 

More specifically, let $\sigma$ be an automorphism of the collinearity graph of $\cI$, and define the composition $\phi^{\sigma}=\sigma \circ \phi$. Since $\sigma$ is a automorphism, $\phi^{\sigma}$ is also an isomorphism from the quotient graph $\hat{B}$ to the collinearity graph of $\cI$. Thus, $\Gamma^{\sigma} := \G(\cH,\cI,\phi^{\sigma})$ is also strongly regular and cospectral with $\Gamma$.

\begin{corollary}\label{cor:newgrssigma0} For any automorphism $\sigma$ of the collinearity graph of $\cI$, $\Gamma^{\sigma}$ is a strongly regular graph cospectral with $\Gamma$.
\end{corollary}

\begin{proof}
  This follows from Theorem \ref{thm:characterizewithscheme}.
\end{proof}

The next question is whether the new graphs are really new, or just isomorphic copies of the original one.

To describe the change from $\G$ to $\G^{\sigma}$ more concretely is that a block $\ell \in \cL$ is now adjacent to all vertices of the fibre $W$ if $\sigma(\phi(W))$ is incident to $\ell$ in $\cI$. Thus, essentially, the fibres are rearranged in how they are connected to the first subconstituent (and adjacencies within the subconstituents remain the same). This implies that if $\sigma$ (as a permutation of points of $\cI$) is such that blocks are mapped to blocks, then essentially nothing changes to $\G$. More formally, we have the following.


\begin{lemma}\label{lem:cosets}
Let $\Aut(\phi(\hat{B}))$ be the automorphism group of the collinearity graph of $\cI$. Let $G < \Aut(\phi(\hat{B}))$ be the collineation group of $\cI$. If $\sigma \in G$, then $\Gamma^{\sigma}$ is isomorphic to $\Gamma$. Further, if 
$\sigma$ and $\tau$ are in the same double coset of $G$, then $\Gamma^{\sigma}$ is isomorphic to $\Gamma^{\tau}$. 
\end{lemma}


\begin{proof}
  From $\sigma$, we get a mapping of the vertices of $\Gamma$ to the vertices of $\Gamma^{\sigma}$. This mapping is an isomorphism if $\sigma \in G$, since edges are mapped to edges. If $\sigma,\tau \in G \alpha G$, then $\sigma = g_1 \tau g_2$ for some $g_1,g_2\in G$, and the claim follows.
\end{proof}

\subsection{A cospectral mate for the Hermitian generalized quadrangle}\label{sec:HGQ}

For $s>t \geq 2$, the known generalized quadrangles of order $(s,t)$ with a regular point all have $s=q^2$ and $t=q$, for prime powers $q$. They are the Hermitian quadrangles $H(3,q^2)$, and other duals of generalized quadrangles of type $T_3(O)$, where $O$ is an ovoid \cite[p.~34]{PT}.
When $O$ is not an ellipic quadric (i.e., the dual is not the Hermitian quadrangle), then $q$ must be even. The only known other ovoids are the Tits ovoids, which exist for $q=2^h$, $h$ odd and $h>2$ \cite[p.~35]{PT}.

Bamberg, Monzillo, and Siciliano \cite{BAMBERG2021281} construct an imprimitive $5$-class association scheme on the second subconstituent of 
$H(3,q^2)$, for odd $q$ (we note that for $q=2$, the parameters of their scheme are not feasible). Their scheme is a fission scheme of the above $4$-class scheme, where $B_4$ is fissioned (and hence it has the same quotient scheme). They prove (with reference to a very technical result) that in this case $\hat{B}$ is isomorphic to the bilinear forms graph $H_q(2,2)$. Pavese (personal communication, April 12, 2022) showed that this is the case for all prime powers $q$ using the well-known dual description of the Hermitian polar space.
 
The bilinear forms graph $H_q(2,2)$ is the subgraph of the Grassmann graph $J_q(4,2)$ induced on the vertices that intersect a fixed plane ($2$-space) $\pi$ trivially (thus, this strongly regular graph is itself the second subconstituent of a strongly regular graph). It has automorphism group $\PGaL(4,q)_{\pi}\cdot 2$ (where subscript-$\pi$ denotes the  stabilizer of $\pi$). It has two types of maximal cliques (like the Grassmann graph), every edge being in exactly one of each type. One kind consists of the planes containing a given line, and the other kind consists of the planes contained in a given $3$-dimensional space. The index-two subgroup $\PGaL(4,q)_{\pi}$ fixes the type of a clique, whereas its (nontrivial) coset interchanges the two types of cliques. For details, see \cite[\S 9.5A]{BCN}. 

Let us now indeed consider the collinearity graph of $H(3,q^2)$, and view it as $\G = \G(\cH,\cI,\phi)$.
Then $\cI$ is isomorphic to the incidence relation of vertices and one type of cliques of the bilinear forms graph $H_q(2,2)$.
Let $\sigma$ be an automorphism of the corresponding bilinear forms graph $H_q(2,2)$ that fixes the type of a clique. Then $\sigma$ is a collineation, so by Lemma \ref{lem:cosets}, 
$\G^{\sigma}$ is isomorphic to $\G$.

On the other hand, if $\sigma$ is an automorphism that interchanges the two types of cliques, then the vertices in the second subconstituent $\G^{\sigma}_2(v)$ are not contained in any maximal clique (of size $q^2+1$) because the other type of cliques in the quotient graph do not correspond to lines of the generalized quadrangle. It follows that we can construct one, and by Lemma \ref{lem:cosets} only one, new graph in this way.

\begin{prop}
The Hermitian generalized quadrangle $H(3,q^2)$ has a cospectral strongly regular graph that is not geometric.
\end{prop}

\section{A problem by Gardiner, Godsil, Hensel, and Royle}
\label{sec:2nbhdextensions}
 
We will next focus on the case that $s=t$, so that the second subconstituents with respect to a regular point are antipodal distance-regular covers of complete graphs. This case was studied already by Gardiner, Godsil, Hensel, and Royle \cite{GARDINER1992161}. In particular, they asked the question which strongly regular graphs have a vertex such that the second subconstituent is an antipodal distance-regular cover.
 
\subsection{Partial solution of the Gardiner et al. problem}

Gardiner et al.~\cite{GARDINER1992161} characterized strongly regular graphs where {\em all} vertices have second subconstituents being antipodal distance-regular graphs. These consist of the Moore graphs, the complements of the triangular graphs, and certain other more interesting complements of so-called semi-partial geometries, such as the collinearity graph of the symplectic generalized quadrangle $W(q)$.

Suppose now that $\G$ is a strongly regular $(\nu,k,\lambda,\mu)$-graph with {\em at least one} vertex such the second subconstituent with respect to this vertex $x$ is an antipodal distance-regular $r$-cover of a complete graph $K_n$ with (generic) intersection array  $\{n-1,n-2-a_1, 1; 1, c_2, n-1\}$ and distinct eigenvalues $n-1,-1,\theta_1,\theta_2$. It follows from arguments such as in the proof of Lemma \ref{lem:spectrum} that $\theta_1$ and $\theta_2$ are also the nontrivial eigenvalues of $\G$, and that the only possible eigenvalues of the first subconstituent $\G_1(x)$ are $\lambda$, $\theta_1$, $\theta_2$, and $\theta_1+\theta_2+1$, where that latter has multiplicity $n-1$ (unless it equals $\lambda$, in which case the joint multiplicity equals $n$). Moreover, $\theta_1+\theta_2=a_1-c_2$ and $n-2-a_1=(r-1)c_2$. Using this, we can characterize some of the above examples (assuming just one second subconstituent being a distance-regular cover). 

Indeed, it follows easily that if $\lambda=0$, then $\G$ is a Moore graph. If $\theta_1=1$, then the second subconstituent is a complete bipartite graph minus a matching, and $\G$ must be the complement of a triangular graph. If $\G_1(x)$ is a disjoint union of cliques, then it follows (by using the standard equations for the parameters such as the above) that $\G$ is pseudo-geometric with the parameters of the collinearity graph of a generalized quadrangle $\GQ(\sqrt{n},\sqrt{n})$.

In the latter case, we can apply our main theorem to the case that $s=t$ to give a  characterization of all graphs containing at least one regular point in this setting, thus partially resolving an open problem in \cite{GARDINER1992161}. 

Throughout this section, we will thus consider a strongly regular graph $\Gamma$ with parameters \[ (q^3 +q^2 + q + 1, q^2 + q, q-1, q +1),\] with a regular point $v$; in this case, $\Gamma_2(v)$ is a distance-regular, antipodal $q$-cover of $K_{q^2}$, whose intersection array is \[ \{q^2 -1, q^2 -q, 1; 1, q, q^2-1\}.\] 
See Figure \ref{fig:distpart-wq} for a depiction of the distance partition of $\Gamma$ with respect to $v$. 

\begin{figure}[htb]
    \centering
    \includegraphics{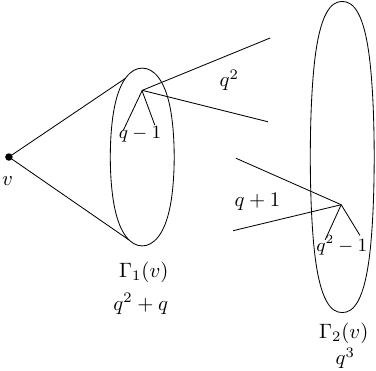} \hspace{20pt}
       \includegraphics{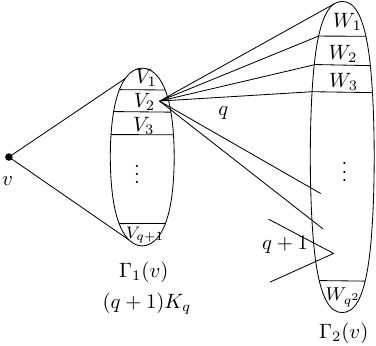}
    \caption{\textit{Left: } The distance partition of $\Gamma$ with respect to $v$. \textit{Right: } The refined distance partition of $\Gamma$ with respect to the regular point $v$. 
    \label{fig:distpart-wq} }
\end{figure}

By applying our main theorem, we obtain the following corollary: 

\begin{corollary}\label{thm:2ndnbhd}
$\Gamma$ is a strongly regular graph with parameters $(q^3 +q^2 + q + 1, q^2 + q, q-1, q +1)$ with a regular point if and only if $\G = \G(\cH,\cI,\phi)$, where
\begin{enumerate}[(i)]
  \item  $\cH$ is a distance regular graph\footnote{Formally, $\cH$ is a metric association scheme generated by a distance-regular graph} with intersection array
$\{q^2 -1, q^2 -q, 1; 1, q, q^2-1\} $; 
\item $\cI$ is an affine plane of order $q$;
\item $\phi$ is a bijection from the fibres of $\cH$ to the points of $\cI$.
\end{enumerate}
\end{corollary}

In some sense, this answers the question by Gardiner et al. \cite{GARDINER1992161}. In the next section, we will however use this characterization to construct the actual examples they asked for.

\subsection{Pseudo-geometric graphs for the symplectic quadrangle}
\label{sec:construction}

Using Corollary \ref{thm:2ndnbhd}, we can now give constructions for new strongly regular graphs cospectral with a given strongly regular graph $\G$ with a regular point $v$ as follows: we may find $\cH,\cI,\phi$ as in Corollary \ref{thm:2ndnbhd} and obtain $\widetilde{\G} = \G(\cH,\cI,\phi')$ where $\phi'$ is obtained from $\phi$ by composition with any permutation. The new graph $\widetilde{\G}$ will be strongly regular with the same parameters as $\Gamma$, but not necessarily isomorphic to $\Gamma$, similar as before. In some simple cases, we will indeed prove that the new graphs we get are not isomorphic to the original graph. In Section \ref{sec:computations}, we use this construction and give some computational results on new graphs cospectral to $W(4)$, and $W(5)$.

Indeed, let $\G = \G(\cH,\cI,\phi)$, let $\sigma$ be any permutation of the points of $\cI$, let $\phi'=\sigma \circ \phi$, and finally $\G^{\sigma}=\G(\cH,\cI,\phi')$, as before. 

\begin{corollary}\label{cor:newgrssigma} For any permutation $\sigma$, $\Gamma^{\sigma}$ is a strongly regular graph cospectral with $\Gamma$.
\end{corollary}

\begin{proof}
This follows from Corollary \ref{thm:2ndnbhd}.
\end{proof}

We noted before that $\sigma$ essentially induces a permutation, $\tau$ say, of the fibres (in how these are connected to the first subconstituent).
We note that if $\tau$ is just a transposition of the fibres $W_a$ and $W_b$, then $\G^{\sigma}$ is obtained from $\G$ by deleting all edges between $W_a \cup W_b$ and $\Gamma_1(v)$, and then connecting each vertex in $W_a$ to $\Gamma_1(w)\cap\Gamma_1(v)$ in $\Gamma$ for $w \in W_b$ and connecting each vertex in $W_b$ to $\Gamma_1(y)\cap\Gamma_1(v)$ in $\Gamma$ for $y \in W_a$. In fact, $\Gamma^{\sigma}$ is thus obtained from $\Gamma$ by  Godsil-McKay switching: take $D=\{v\} \cup \Gamma_1(v)$, $C_1=W_a\cup W_b$, and the remaining fibres of $H$ as the other ($q^2-2$) sets $C_i$. Then the assumptions of Lemma \ref{lem:GM} are satisfied, and $\G^{\sigma}$ and $\G$ are cospectral. 
Since any permutation can be written as a product of transpositions, and we can apply Godsil-McKay switching over and over (indeed, similar partitions remain to satisfy the required conditions), the above corollary likewise follows.

For $q \geq 3$, we will show now that for different kinds of permutations, we obtain different graphs. Fix a line $\ell$ of the affine plane, and let $\sigma$ be a permutation that (only) permutes $r$ points of $\ell$ (for $r=2,\dots,q$). Let us call a clique of size $q+1$ a maxclique. In the original generalized quadrangle, these are its lines. By counting for each vertex the number of incident maxcliques, it will follow that for different $r$, the obtained graphs are non-isomorphic.

Indeed, $v$ remains on $q+1$ maxcliques, as do all $q + (q-r)q$ vertices in $\G_1(v)$ that correspond to lines of the affine plane that are parallel to $\ell$, or that are incident to one of the $q-r$ fixed points of the line $\ell$. The other $rq$ vertices in $\G_1(v)$ remain on just $1$ maxclique. The $qr$ vertices in the fibres that are permuted also remain on just $1$ maxclique, the $q(q-r)$ vertices in the other fibres incident to $\ell$ remain on $q+1$ maxcliques, and finally, the $q(q^2-q)$ vertices in all other fibres remain on $q+1-r$ maxcliques. Thus, in total there are $2rq$ vertices that are on $1$ maxclique, $2q^2-(2r-1)q+1$ vertices that are on $q+1$ maxcliques, and $q^3-q^2$ vertices that are on $q+1-r$ maxcliques.

It is clear that the counting gets more technical for other kinds of permutations. One observation we would like to make is that for a permutation of three points not on a line, the $3q$ vertices in the corresponding fibres are no longer on any maxclique. Thus, we obtain yet another different graph, and conclude the following.

\begin{theorem}\label{thm:Wqmates} For $q \geq 3$, the symplectic generalized quadrangle $W(q)$ has at least $q$ cospectral strongly regular graphs that are not geometric. 
\end{theorem}

\section{Computations} \label{sec:computations}

Using Corollary \ref{cor:newgrssigma} and an existing strongly regular graph $\Gamma$ with a regular point, as in Section \ref{sec:2nbhdextensions}, we can construct new graphs that are cospectral with, but not necessarily isomorphic to $\Gamma$. Using SageMath \cite{sage}, we have done this for the collinearity graph of $W(q)$, for $q=2,3,4,5$, and for the collinearity graphs of $\GQ(4,2)$, $\GQ(9,3)$, and $\GQ(16,4)$. 

For $q=2,3$, the parameter classes containing $W(q)$ have been fully enumerated
by Spence \cite{Spe2000},
and, though we do not obtain previously unknown graphs, the computation still gives some insight into this operation on graphs with a regular point. For $q=4$, we computed $133\,718$  new isomorphism classes of strongly regular graphs with parameters $(85,20,3,5)$, cospectral with the collinearity graph of $W(4)$. For $q=5$, we computed $27\,298$ new isomorphism classes of strongly regular graphs with parameters $(156,30,4,6)$, cospectral with $W(5)$. We will now summarize the computations that we did on each class. 

For $q=2$, there is a unique generalized quadrangle $\GQ(2,2)$ coming from $W(2)$; its collinearity graph $\Gamma$ is a strongly regular graph with parameters $(15, 6, 1,3)$ and is the unique graph in the class, up to isomorphism. Here, $\Gamma_1(v)$ is three copies of $K_2$ and $\Gamma_2(v)$ is isomorphic to the cube. The affine plane has 6 lines and 4 points; the bipartite incidence graph has 10 vertices and is biregular with degrees 2 and 3, so it is isomorphic to the graph obtained from $K_4$ by subdividing all edges exactly once. Every permutation of the points induces an automorphism of the affine plane graph, thus all $\Gamma^{\sigma}$ are isomorphic to $\Gamma$.

For $q=3$, there are 2 generalized quadrangles, $W(3)$ and its dual $T_2(O)$, both of which are vertex-transitive. Every point of $W(3)$ is a regular point and no point of $T_2(O)$ is a regular point. In the collinearity graph of $W(3)$, $\Gamma_1(v)$ consists of four copies of $K_3$ and $\Gamma_2(v)$ is a graph we will call $H_1$, whose graph6 string is as follows:
\begin{verbatim}
    ZQhSaQOQcaHC``WWObJCG`WD_Y????FwC?B~BBB_]@c@eWOKKKcKoB`E@wP?
\end{verbatim}
The affine plane gives a graph on $21$ vertices; it has $12$ lines and $9$ points and its 
automorphism group $G$
has size $432$. To find all possible $\G^{\sigma}$, we computed the $9$ double cosets of $G \backslash \Sym(9) / G$ and found $9$ non-isomorphic graphs, including the collinearity graph of $W(3)$. The computation times are summarized in Table \ref{tab:computationtimes3}.

\begin{table}[htpb]\label{tab:computationtimes3}
\centering
\begin{tabular}{ll}
Computational task                                   & Computation time \\
\hline 
construct $\Gamma^{\sigma}$ for random $\sigma$      & 32 ms            \\
compute canonical labelling of one $\Gamma^{\sigma}$ & 2.39 ms          \\
compute the double cosets $G \backslash \Sym(9) / G$   & 0.01s            \\
total computation time                               & 0.18s
\end{tabular}
\end{table}

Exactly one of these 9 graphs has a second regular point, where the second subconstitutuent is not isomorphic to $H_1$. The graph6 string of the second subconstituent is 
\begin{verbatim}
    Z@G`PG_GH_HKRGPbGXAOoSgDSOaacCKdaBSK`JDCOxCOLHWcgpAf?SgRGTA?
\end{verbatim}
Repeating the computation gives $9$ distinct graphs, $8$ of which are distinct from the $9$ obtained from the collinearity graph of $W(3)$. 

In total, we construct 17 graphs with at least one regular point, one of which is the collinearity graph of $W(3)$. This parameter class SRG$(40,12,2,4)$ was completely enumerated by Spence \cite{Spe2000} and contains 28 isomorphism classes of graphs. None of the remaining 11 graphs have a regular point. 

For $W(4)$ and $W(5)$, the number of double cosets of $G \backslash \Sym(q^2) / G$ exceeded GAP's memory capacity, so a complete enumeration was not feasible. Nevertheless, we have many new classes of non-isomorphic graphs.

For $W(4)$, to give an idea of the computation time, we sampled $10\,000$ elements of $\Sym(16)$ at random; this resulted in $9\,887$ double coset representatives from different double cosets, and $9\,812$ pairwise non-isomorphic graphs found in 2h 56min 41s. In total, we computed $133\,718$  new isomorphism classes of strongly regular graphs with parameters $(85,20,3,5)$.

For $W(5)$, we also sampled $10\,000$ elements of $\Sym(25)$ at random; this resulted in $9\,969$ double coset representatives from different double cosets, and $9\,968$ pairwise non-isomorphic graphs found in 20h 31min 32s. In total, we computed $27\,298$ new isomorphism classes of strongly regular graphs with parameters $(156,30,4,6)$.

The new graphs, as well as the previously known generalized quadrangles, that we found are available in \cite{newgrs}. 

For $s>t$, we also computed this switching operation for $\GQ(4,2)$, $\GQ(9,3)$, and $\GQ(16,4)$. 

There are $78$ strongly regular graphs with parameters $(45,12,3,3)$ \cite{SRG4512}, one of which is the collinearity graph of $\GQ(4,2)$. Of the $78$ graphs, there are 13 graphs each of which has exactly one vertex, up to actions of the automorphism group, such that its first subconstituent consists of three cliques of order $4$. Of these, six graphs have a regular point. All quotient graphs of the second subconstituents of these six graphs are isomorphic to the complement of the lattice graph. The six graphs come in pairs that can be obtained from each other by our construction. A perhaps surprising consequence of the fact that only covers of the complement of the lattice graph occur, is that the complement of the Shrikhande graph does not have the required $2$-cover scheme. Note that this graph is the collinearity graph of a net, as it comes from the cyclic Latin square.

The second neighbourhood of the point of graph of the unique generalized quadrangle $\GQ(9,3)$ has $243$ vertices and is a $3$-fold cover of the bilinear forms graph, which is strongly regular with parameters $(81, 32, 13,12)$, and has an automorphism group of order $186\,624$. 
We find one graph with parameters $(280,36,8,4)$ which is cospectral to, but not isomorphic to the collinearity graph of $\GQ(9,3)$, as expected, from the work in Section \ref{sec:HGQ}. 

The second neighbourhood of the point of graph of the unique generalized quadrangle $\GQ(16,4)$ has $1\,024$ vertices and is a $4$-fold cover of a strongly regular graph with parameters $(256, 75, 26, 20)$, 
and whose automorphism group has order $11\,059\,200$ and is isomorphic to the bilinear forms graph. We find one graph with parameters $(1105,80,15,5)$ which is cospectral to, but not isomorphic to the collinearity graph of $\GQ(16,4)$. Each vertex of $\GQ(16,4)$ lies on $5$ maximum cliques of order $17$, but, in the cospectral mate we constructed, there is a vertex $v$ such that $v$ is on $5$ maximum cliques of order $17$, the $80$ neighbours of $v$ lie on exactly one clique of order $17$ and the remaining $1\,024$ vertices in the second subconsituent of $v$ lie on no maximum cliques, as one would expect from Section \ref{sec:HGQ}. 

\section{Further observations}\label{sec:end}

\subsection{Another characterization of regular points}

Even though our definition of a regular point in a pseudo-geometric strongly regular graph led to our main results, it is perhaps somewhat unnatural (or at least, non-elementary). We will next give a characterization that is more in line with the definition of regular points in generalized quadrangles. 

\begin{theorem}\label{thm:simplescharacterization} A vertex $v$ in a pseudo-geometric strongly regular graph with parameters $((ts+1)(s+1),(t+1)s,s-1,t+1)$ is regular if and only if the first subconstituent $\G_1(v)$ is a disjoint union of cliques and the sets $\{v,x\}^{\perp \perp}$ are cocliques of size $t+1$ for all $x$ not adjacent to $v$.
\end{theorem}

\begin{proof} One direction is clear from the results in Section \ref{sec:4class}. For the other direction, let us assume that $\G_1(v)$ is indeed a disjoint union of cliques and the sets $\{v,x\}^{\perp \perp}$ are cocliques of size $t+1$ for all $x$ not adjacent to $v$, and then show that we obtain an association scheme on the second subconstituent (with the right parameters). Note first that the sets $\{v,x\}^{\perp \perp} \setminus \{v\}$ partition the vertex set of the second subconstituent (because having the same  common neighbors with $v$ is an equivalence relation); indeed, these will be the fibres of the scheme that is generated by $\G_2(v)$. As before, we can now define a natural incidence relation $\cI$ between the fibres (points $\cP$) and the vertices in $\G_1(v)$ (blocks $\cL$). Because two vertices in the same clique in $\G_1(v)$ have already $s-1$ common neighbors in that clique, they share no neighbor in $\G_2(v)$, so as blocks of $\cI$ they are parallel. It also implies that each vertex in $\G_2(v)$ has precisely one neighbor from each of the cliques in $\G_1(v)$. Because $c=t+1$, it also follows that two blocks from different cliques meet in precisely one fibre, hence $\cI$ is a $(t+1)$-net of order $s$.

Next, we can define the relations $B ,B_2,B_3,B_4$ of the (putative) scheme, using the induced adjacency matrix $B$ of $\G_2(v)$, the fibres ($B_3$) and the incidence relation $\cI$. Indeed, two non-adjacent vertices (from different fibres) are related in $B_4$ if the corresponding fibres are not collinear in $\cI$, and $B_2=J-I-B-B_3-B_4$ is the remaining relation (nonadjacent in collinear fibres).
Note that because $c=t+1$, it follows that vertices in the same fibre have $p^3_{11}=0$ common neighbors. Two vertices that are $B_4$-related have $p^4_{11}=t+1$ common neighbors (as they have no common neighbor in the first subconstituent), and similarly $p^2_{11}=t$.

Now the crucial question is whether $B$ corresponds to collinearity like $B_2$ does, or in other words, whether $B$ is a cover of the collinearity graph, $\hat{B}$ say, of $\cI$. To answer this question in the affirmative, we first note that two adjacent vertices in the second subconstituent have either 0 or 1 common neighbor in the first subconstituent, and hence $s-1$ or $s-2$ in the second subconstituent. By Lemma \ref{lem:spectrum}, we know the spectrum of $B$, and this implies that we know the number of triangles in $B$. This implies that every edge must be in $s-2$ triangles, so $p^1_{11}=s-2$, but more importantly, $B$ is indeed a cover of the collinearity graph of $\cI$. 

Note that we have in fact derived \eqref{eq:B^2}. That the remaining intersection numbers are well-defined follows from the ones that we already have in combination with the quotient scheme on the fibres.
\end{proof}

\subsection{Constructing cospectral mates of \texorpdfstring{$\GQ(q-1,q+1)$}{GQ(q-1,q+1)} }
\label{sec:q-1q+1}

By introducing the concept of a regular point in a generalized quadrangle of order $(q,q)$, Payne \cite{PAYNE1971201} constructed generalized quadrangles of order $(q-1,q+1)$, with a spread, on the points noncollinear with the regular point (see also \cite[3.1.4]{PT}). Brouwer \cite{BROUWER1984101} already observed that removing a spread (i.e., a partition of the vertices into $(s+1)$-cliques) from a pseudo-geometric strongly regular graph with parameters $((s+1)(st+1), s(t+1), s-1,t+1)$ gives a distance-regular antipodal $(s+1)$-cover of a complete graph on $st+1$ vertices (see also \cite[Prop. 12.5.2]{BCN}), and the other way around. 

Because our characterization of regular points in Corollary \ref{thm:2ndnbhd} gives rise to an antipodal $q$-cover of a complete graph on $q^2$ vertices, by adding edges to the fibres of such a cover to change the cocliques into cliques (adding the spread), we obtain pseudo-geometric strongly regular graphs with the same parameters as the collinearity graph of a $\GQ(q-1,q+1)$. Of course, if we start from a regular point in a $\GQ(q,q)$, or even from the same regular point in one of our new pseudo-geometric graphs, we obtain nothing new. However, in the new graphs, there might also exist other regular points, whose second subconstituents may give other distance-regular antipodal covers, and other pseudo-geometric strongly regular graphs.

For $W(3)$, this process gives both spreads of $\GQ(2,4)$, and hence both distance-regular antipodal $3$-covers of $K_9$. Note that $\GQ(2,4)$ itself is determined as a strongly regular graph by its parameters. From $W(4)$, we obtain $6$ strongly regular graphs with parameters $(64,18,2,6)$ (with a spread). By a different computation, in SageMath \cite{sage}, we checked that of the 167 strongly regular graphs with parameters $(64,18,2,6)$ \cite{HaemersWillemH2001TPGf}, there are $16$ that have a spread (and which, in turn, can be used to construct pseudo-geometric graphs with the same parameters as $W(4)$). From $W(5)$, we obtain one new strongly regular graph with parameters $(125,28,3,7)$, which is not isomorphic to $\GQ(4,6)$, or any of the other previously known strongly regular graphs with these parameters. The new strongly regular graph  can be found in the author's code repository at \cite{newgrs}. 

We finally note that removing spreads in strongly regular graphs can be done in greater generality to obtain imprimitive $3$-class association schemes, as was shown by Haemers and Tonchev \cite{HaemersTonchev1996}. Unfortunately, this does not apply to the schemes in Theorem \ref{thm:characterizewithscheme} (with $s>t$).

\subsection{Geometric and pseudo-geometric graphs} 

The point graph of a partial geometry $\pg(s,t,\alpha)$ is a strongly regular graph with parameters
\begin{equation}\label{eq:pgSRG}
    \left( (s+1)\left(\frac{st}{\alpha}+1 \right), s(t+1), s-1 + t(\alpha-1) , \alpha (t+1) \right).
\end{equation}
A strongly regular graph which is the point graph of a partial geometry is said to be \textsl{geometric}. A strongly regular graph with parameters as given in \eqref{eq:pgSRG}, for integers $s$ such that $1\leq \alpha \leq \min \{s, t+1\}$ and $\alpha$ divides $st(s+1)$, is said to be \textsl{pseudo-geometric}. There are many results which determine whether or not certain classes of  pseudo-geometric strongly regular graph are geometric, for example by Bose \cite{Bose1963}, who first raised the problem. See also \cite{BrouwervanLint}. 

A generalized quadrangle $\GQ(s,t)$ is a partial geometry $\pg(s,t,1)$. For prime powers $q$, the following list consists of the values for $(s,t)$ for which a $\GQ(s,t)$ is known to exist: 
\[
(1,q), (q,1), (q-1,q+1), (q+1,q-1), (q,q), (q,q^2), (q^2,q), (q^2, q^3), (q^3,q^2). 
\]
Cameron, Goethals, and Seidel \cite{CAMERON1978257} showed that every  strongly regular graph with the same parameters as $\GQ(q,q^2)$ is geometric. By the results in this paper, there exist pseudo-geometric but not geometric graphs cospectral to $\GQ(q^2,q)$ and $\GQ(q,q)$, for all prime powers $q$. We also exhibit a sporadic example of a pseudo-geometric but not geometric graph cospectral to $\GQ(q-1,q+1)$, for $q=5$. The point graphs $\GQ(1,q)$ are imprimitive and thus are determined by their spectrum (and hence geometric). The point graph of $\GQ(q,1)$ is known as a grid graph and is also determined by its spectrum, except for $q=3$, in which case we have the (non-geometric) Shrikhande graph \cite{Shrikhande10.1214/aoms/1177706207}. Cossidente and Pavese \cite{CosPav2016} construct 
pseudo-geometric, but not geometric, strongly regular graphs having the same parameters as the point graph of a $\GQ(q^2, q^3)$, when $q$ is a prime power, and of a $\GQ(q, q)$, when $q$ an even power of a prime. 

Wallis \cite{Wallis} constructed non-geometric strongly regular graphs with the same parameters of a $\GQ(q+1, q-1)$, for all prime powers $q$. Fon-Der-Flaass rediscovered this construction \cite[Constr.~1]{FDF2002}; on top of this he constructed strongly regular graphs with the same parameters as a  $\GQ(q,q)$ \cite[Constr.~2]{FDF2002} or a $\GQ(q-1,q+1)$ \cite[Constr.~3]{FDF2002} (for all prime powers $q$). It is not hard to see that in all three constructions, non-geometric examples can be obtained whenever (the ``line size'') $s+1$ is at least $4$. We also note that in Fon-Der-Flaass' second construction, the so-called new vertices are regular points.


What remains is the question whether there exist non-geometric strongly regular graphs with the same parameters as a $\GQ(q^3,q^2)$, for infinitely many values of $q$.


\subsection{Quasi-regular points}

A quasi-regular point in a generalized quadrangle $\GQ(s,t)$ with $s<t$ is a point $v$ such that each triad $(v,x,y)$ of noncollinear points has $|\{v,x,y\}^{\perp}\}|$ either $q+1$ or $0$ (for some $q$).

This clearly defines two relations on noncollinear points in the second subconstituent $\Gamma_2(v)$, and indeed one obtains a $3$-class association scheme on the second subconstituent, as was shown by Payne \cite{Payne1992} (see also \cite{Hobart1993}), if $t<s^2$. 
Payne in fact showed that there is a larger coherent configuration involved.

One may thus wonder whether results analogous to ours could be obtained when defining quasi-regular points in pseudo-geometric graphs.

In the exceptional (extremal) case that $t=s^2$, the perp of each triad has size $q+1$, and the second subconstituent $\Gamma_2(v)$ is a strongly regular graph.
It in fact follows easily from considering the eigenvalues that the second subconstituent with respect to any point is strongly regular. Moreover, by a result of Cameron, Goethals, and Seidel \cite{CAMERON1978257}, in this case every pseudo-geometric graph is a generalized quadrangle, as we saw before. Also here there is an interesting incidence relation between the subconstituents: this is a strongly regular design as defined by Higman \cite{HIGMAN1988411}.


\begin{thebibliography}{99}

\bibitem{BAMBERG2021281}
J. Bamberg, G. Monzillo, and A. Siciliano.
\newblock Pseudo-ovals of elliptic quadrics as {D}elsarte designs of
  association schemes.
\newblock {\em Linear Algebra and its Applications}, 624:281--317, 2021.

\bibitem{Bose1963}
R. Bose.
\newblock Strongly regular graphs, partial geometries and partially balanced
  designs.
\newblock {\em Pacific Journal of Mathematics}, 13(2):389--419, 1963.

\bibitem{BROUWER1984101}
A.E. Brouwer.
\newblock Distance regular graphs of diameter 3 and strongly regular graphs.
\newblock {\em Discrete Mathematics}, 49(1):101--103, 1984.

\bibitem{BCN}
A.E. Brouwer, A.M. Cohen, and A. Neumaier.
\newblock {\em Distance-regular graphs}.
\newblock Springer-Verlag, Berlin, 1989.

\bibitem{Bro1989}
A.E. Brouwer, A.V. Ivanov, and M.H. Klin.
\newblock Some new strongly regular graphs.
\newblock {\em Combinatorica}, 9(4):339--344, 1989.

\bibitem{BrouwervanLint}
A.E. Brouwer and J.H. {van Lint}.
\newblock Strongly regular graphs and partial geometries.
\newblock In D.H. Jackson and S.A. Vanstone, editors, {\em Enumeration and
  Design}, pages 85--122. Academic Press Inc., United States, 1984.

\bibitem{BvMSRG}
A.E. Brouwer and H. Van Maldeghem.
\newblock {\em Strongly regular graphs}.
\newblock Cambridge University Press, 2022.

\bibitem{DECAEN1998559}
{D. de} Caen.
\newblock The spectra of complementary subgraphs in a strongly regular graph.
\newblock {\em European Journal of Combinatorics}, 19(5):559--565, 1998.

\bibitem{CAMERON1978257}
P.J. Cameron, J.M. Goethals, and J.J. Seidel.
\newblock Strongly regular graphs having strongly regular subconstituents.
\newblock {\em Journal of Algebra}, 55(2):257--280, 1978.

\bibitem{SRG4512}
K. Coolsaet, J. Degraer, and E. Spence.
\newblock The strongly regular $(45,12,3,3)$ graphs.
\newblock {\em Electronic Journal of Combinatorics}, 13:R32, 2006.

\bibitem{CosPav2016}
A. Cossidente and F. Pavese.
\newblock Strongly regular graphs from classical generalized quadrangles.
\newblock {\em Designs, Codes, and Cryptography}, 85(3):457--470, 2016.

\bibitem{CrnkovicDean2021Srgw}
D. Crnković, A. Švob, and V.D. Tonchev.
\newblock Strongly regular graphs with parameters (81, 30, 9, 12) and a new
  partial geometry.
\newblock {\em Journal of Algebraic Combinatorics}, 53(1):253--261, 2021.

\bibitem{VANDAM1995139}
{E.R. van} Dam.
\newblock Regular graphs with four eigenvalues.
\newblock {\em Linear Algebra and its Applications}, 226-228:139--162, 1995.
\newblock Honoring J.J.Seidel.


\bibitem{FenMomXia2014}
T. Feng, K. Momihara, and Q. Xiang.
\newblock Constructions of strongly regular {C}ayley graphs and skew {H}adamard
  difference sets from cyclotomic classes.
\newblock {\em Combinatorica}, 35(4):413--434, 2014.

\bibitem{FENG2012982}
T. Feng and Q. Xiang.
\newblock Strongly regular graphs from unions of cyclotomic classes.
\newblock {\em Journal of Combinatorial Theory, Series B}, 102(4):982--995,
  2012.

\bibitem{FDF2002}
D.G. Fon-Der-Flaass.
\newblock New prolific constructions of strongly regular graphs.
\newblock {\em Advances in Geometry}, 2:301--306, 2002.

\bibitem{Gar1974}
A. Gardiner.
\newblock Antipodal covering graphs.
\newblock {\em Journal of Combinatorial Theory. Series B}, 16(3):255--273,
  1974.

\bibitem{GARDINER1992161}
A.D. Gardiner, C.D. Godsil, A.D. Hensel, and G.F. Royle.
\newblock Second neighbourhoods of strongly regular graphs.
\newblock {\em Discrete Mathematics}, 103(2):161--170, 1992.

\bibitem{GHINELLI199587}
D. Ghinelli and S. Löwe.
\newblock Generalized quadrangles with a regular point and association schemes.
\newblock {\em Linear Algebra and its Applications}, 226-228:87--104, 1995.
\newblock Honoring J.J.Seidel.

\bibitem{GR}
C. Godsil and G. Royle.
\newblock {\em Algebraic graph theory}, volume 207 of {\em Graduate Texts in
  Mathematics}.
\newblock Springer-Verlag, New York, 2001.

\bibitem{GMswitching}
C.D. Godsil and B.D. McKay.
\newblock Constructing cospectral graphs.
\newblock {\em Aequationes Mathematicae}, 25:257--268, 1982.

\bibitem{GUO2020103128}
I. Guo, J.H. Koolen, G. Markowsky, and J. Park.
\newblock On the nonexistence of pseudo-generalized quadrangles.
\newblock {\em European Journal of Combinatorics}, 89:103128, 2020.

\bibitem{HaemersTonchev1996}
W.H. Haemers and V.D. Tonchev.
\newblock Spreads in strongly regular graphs.
\newblock {\em Designs, Codes and Cryptography}, 8:145--157, 1996.

\bibitem{HaemersWillemH2001TPGf}
W.H. Haemers and E. Spence.
\newblock The pseudo-geometric graphs for generalized quadrangles of order $(3,
  t)$.
\newblock {\em European Journal of Combinatorics}, 22(6):839--845, 2001.

\bibitem{HIGMAN1988411}
D.G. Higman.
\newblock Strongly regular designs and coherent configurations of type [323].
\newblock {\em European Journal of Combinatorics}, 9(4):411--422, 1988.

\bibitem{Hobart1993}
S.A. Hobart and S.E. Payne.
\newblock Reconstructing a generalized quadrangle from its distance two
  association scheme.
\newblock {\em Journal of Algebraic Combinatorics}, 2:261--266, 1993.


\bibitem{IHRINGER2019464}
F. Ihringer and A. Munemasa.
\newblock New strongly regular graphs from finite geometries via switching.
\newblock {\em Linear Algebra and its Applications}, 580:464--474, 2019.

\bibitem{IHRINGER2021112560}
F. Ihringer, F. Pavese, and V. Smaldore.
\newblock Graphs cospectral with {NU}$(n+1,q^2)$, $n \neq 3$.
\newblock {\em Discrete Mathematics}, 344(11):112560, 2021.

\bibitem{KRCADINAC2021105493}
V. Krčadinac.
\newblock A new partial geometry pg$(5,5,2)$.
\newblock {\em Journal of Combinatorial Theory, Series A}, 183:105493, 2021.

\bibitem{VanLintSchrijver}
{J.H. van} Lint and A. Schrijver.
\newblock Construction of strongly regular graphs, two-weight codes and partial
  geometries by finite fields.
\newblock {\em Combinatorica}, 1:63--73, 1981.

\bibitem{Payne1992}
S.E. Payne.
\newblock Coherent configurations derived from quasiregular points in
  generalized quadrangles.
\newblock {\em Finite Geometry and Combinatorics, London Math. Soc. Lecture
  Note Series}, 191:327--339, 1993.
\newblock Proceedings of Finite Geometry and Combinatorics, Second
  International Conference at Deinze, Belgium, 1992.

\bibitem{PT}
S.E. Payne and J.A. Thas.
\newblock {\em Finite generalized quadrangles}, volume 110.
\newblock European Mathematical Society, 2009.

\bibitem{PAYNE1971201}
S.E. Payne.
\newblock Nonisomorphic generalized quadrangles.
\newblock {\em Journal of Algebra}, 18(2):201--212, 1971.

\bibitem{Shrikhande10.1214/aoms/1177706207}
S.S. Shrikhande.
\newblock {The uniqueness of the $\mathrm{L}_2$ association scheme}.
\newblock {\em The Annals of Mathematical Statistics}, 30(3):781--798, 1959.

\bibitem{Spe2000}
E. Spence.
\newblock The strongly regular (40, 12, 2, 4) graphs.
\newblock {\em Electronic Journal of Combinatorics}, 7:R22, 2000.

\bibitem{sage}
{The Sage Developers}.
\newblock {\em {S}ageMath, the {S}age {M}athematics {S}oftware {S}ystem
  ({V}ersion 8.3)}, 2018.
\newblock {\tt http://www.sagemath.org}.

\bibitem{Wallis}
W.D. Wallis.
\newblock Construction of strongly regular graphs using affine designs.
\newblock {\em Bulletin of the Australian Mathematical Society}, 4(1):41--49,
  1971.

\bibitem{newgrs}
\newblock {
New pseudo-generalized quadrangles}, 
\newblock \doi{10.5281/zenodo.6596987}, May 2022.


\end{thebibliography}

\end{document}